\documentclass[10pt]{amsart}

\usepackage{amsmath} \usepackage{amssymb, amscd,cancel, graphicx,soul}

	\usepackage{pinlabel,mathtools}

\usepackage{hyperref}
\hypersetup{
    colorlinks=true, %set true if you want colored links
    linktoc=all,     %set to all if you want both sections and subsections linked
    linkcolor=black,  %choose some color if you want links to stand out
}

\usepackage{enumerate}

\usepackage{cleveref}

\usepackage{tikz-cd}

\usepackage{mathdots}
\newtheorem{thm}{Theorem}[section]
\newtheorem{letteredthm}{Theorem}

\newtheorem{cor}[thm]{Corollary}
\newtheorem{lem}[thm]{Lemma}
\Crefname{lem}{Lemma}{Lemmas}
\newtheorem{prop}[thm]{Proposition}

\newtheorem*{ques*}{Question}

\newtheorem*{example*}{Example}
\newtheorem*{porism*}{Porism}
\newtheorem{scholium}[thm]{Scholium}

\newtheorem*{thm*}{Theorem}
\newtheorem*{defin*}{Definition}
\newtheorem*{lem*}{Lemma}
\newtheorem*{prop*}{Proposition}
\newtheorem*{remark*}{Remark}

\def\cA{{\mathcal A}}

\def\cS{{\mathcal S}}

\def\JF{{\rm JF}}
\def\JFC{{\rm JFC}}

\def\bC{{\mathbb C}}

\def\bF{{\mathbb F}}

\def\cL{{\mathcal L}}
\def\cR{{\mathcal R}}

\def\bR{{\mathbb R}}

\def\bZ{{\mathbb Z}}

\def\Crit{{\mathrm{Crit}}}

\def\Emb{{\mathrm{Emb}}}

\def\del{{\partial}}

\def\graph{{\textup{graph}}}

\def\ev{{\mathrm{ev}}}

\newcommand{\into}{\hookrightarrow}

\begin{document}
\thispagestyle{empty}
\title[Jordan curves inscribe a positive measure of rectangles]{Jordan curves inscribe a positive measure of rectangles}
\author{Joshua Evan Greene} 
\address{Department of Mathematics, Boston College, USA}
\email{joshua.greene@bc.edu}
\urladdr{https://sites.google.com/bc.edu/joshua-e-greene}
\author{Andrew Lobb} 
\address{Mathematical Sciences,
	Durham University,
	UK}
\email{andrew.lobb@durham.ac.uk}
\urladdr{http://www.maths.dur.ac.uk/users/andrew.lobb/}
\thanks{JEG was supported by the National Science Foundation under Award No.~DMS-2304856 and by a Simons Fellowship.}

%%%
%%%
%%%

\begin{abstract}
Suppose that $\gamma \subset \bC$ is a Jordan curve of diameter $2R$ which encloses a region of area $A$.
We prove that there exists a subset $I \subset (0,\pi)$ of measure at least $A/R^2$ such that if $\theta \in I$, then there exist four points on $\gamma$ at the vertices of a rectangle whose diagonals meet at angle $\theta$.
\end{abstract}

\maketitle

%%%
%%%
%%%

\section{Introduction}
\label{sec:intro}
A polygon $P \subset \bC$ \emph{inscribes} in a Jordan curve $\gamma \subset \bC$ if there exists an orientation-preserving similar copy of $P$ whose vertex set is contained in $\gamma$.
The case of a quadrilateral $P$ is distinguished, because a generic differentiable curve $\gamma$ admits a finite set of inscriptions of $P$.
In this case, it is interesting to determine whether this set is nonempty.
The original problem in this domain is the Square Peg Problem due to Toeplitz, which asks whether every Jordan curve $\gamma \subset \bC$ inscribes a square \cite{toeplitz1911}.  It is still unsolved.

For Jordan curves possessing additional regularity, much is now known.
Schnirelmann used a bordism argument to affirmatively solve the Square Peg Problem for differentiable curves \cite{schnirelman1929}.
We used symplectic techniques to prove the optimal result that every quadrilateral that inscribes in a circle also inscribes in every differentiable Jordan curve \cite{greenelobb1,greenelobb2}.
On the other hand, the only quadrilaterals which inscribe in the circle and in all triangles are isosceles trapezoids \cite{pak}.
The Trapezoidal Peg Problem for continuous curves asks whether every isosceles trapezoid inscribes in every Jordan curve.
Intermediate between it and Toeplitz's original question, the Rectangular Peg Problem for continuous curves asks whether every rectangle inscribes in every Jordan curve.
Asano and Ike used microlocal sheaf techniques to affirmatively solve the Rectangular Peg Problem for rectifiable curves \cite{asanoike}.

For Jordan curves without any regularity, much less is known.
Vaughan gave a proof {\em from the book} that every Jordan curve inscribes a rectangle \cite{meyerson1981}.
Schwartz improved this result by showing that all but at most four points on every Jordan curve are vertices of inscribed rectangles in the curve \cite{schwartz2020}.
The result is sharp, by considering the case of a noncircular ellipse.
The techniques used to prove these results are not symplectic, and they do not detect the similarity classes of the inscribed rectangles.

The purpose of this note is to obtain progress on the Rectangular Peg Problem for continuous curves using symplectic techniques.
To state our main result, we introduce some notation.
For an angle $0 < \theta < \pi$, a {\em $\theta$-rectangle} is a rectangle in the plane whose diagonals make angle $\theta$ with one another.
For a Jordan curve $\gamma \subset \bC$, let $A(\gamma)$ denote the area enclosed by $\gamma$, let $R(\gamma)$ denote the radius (half of the diameter) of $\gamma$, and let $B(\gamma) = A(\gamma) / R(\gamma)^2$.
The notation honors Bieberbach: the 2-dimensional case of his Isodiametric Inequality reads $0 < B(\gamma) \le \pi$, and $B(\gamma) = \pi$ if and only if $\gamma$ is a circle \cite{Bieberbach1915}.

\begin{letteredthm}
	\label{thm:mainthm}
	If $\gamma \subset \bC$ is a Jordan curve, then there exists a subset $\cS(\gamma) \subset (0, \pi)$ of measure at least $B(\gamma)$ such that for all $\theta \in \cS(\gamma)$, $\gamma$ inscribes a $\theta$-rectangle.
\end{letteredthm}

The proof of \Cref{thm:mainthm} is based on our earlier construction of Jordan Floer homology \cite{greenelobb3,greenelobb4}.
In particular, we establish and use some properties of the spectral invariants that are defined in that theory.
Previously, we used properties of the spectral invariants to prove \Cref{thm:mainthm} under the assumption that $\gamma$ is rectifiable, with the added conclusion that $\cS(\gamma)$ can be chosen to be an interval \cite[Theorem A]{greenelobb3}.
Similarly, spectral invariants are used in the microlocal sheaf approach to solve the Rectangular Peg Problem for rectifiable curves \cite[Corollary 1.2]{asanoike}.
We give a short proof of the following result, intermediate in strength between these two results:

\begin{letteredthm}
	\label{thm:otherthm}
	If $\gamma \subset \bC$ is a rectifiable Jordan curve, then $\gamma$ inscribes a $\theta$-rectangle for all $0 < \theta < B(\gamma)$.
\end{letteredthm}
% check \le vs <

\noindent
\Cref{thm:otherthm} implies a positive solution to the Rectifiable Peg Problem for every rectifiable Jordan curve $\gamma$ with $B(\gamma) > \pi/2$, i.e.~which encloses more than half of the area of a circle of equal radius.

The remainder of the paper is organized as follows.
\Cref{sec:thms} contains the proofs of Theorems \ref{thm:mainthm} and \ref{thm:otherthm}.
The proof of \Cref{thm:otherthm} relies on the Triangle Inequality for spectral invariants, which already appeared in the literature.
The proof of \Cref{thm:mainthm} relies on a Derivative Property for spectral invariants, which is new.
It relates the derivative of the spectral invariant to the diameter of an inscribed rectangle.
We also state a few refinements of \Cref{thm:mainthm} which follow from its proof.
\Cref{sec:genericity} describes the result that for an open, dense set of smooth Jordan curves, the space of inscribed rectangles and binormals has the structure of a smooth, 1-dimensional, closed manifold.
It is an analogue for smooth Jordan curves of a result of Schwartz for PL Jordan curves \cite{schwartz2020}.
Its proof is given in the Appendix.
\Cref{sec:action} equates a derivative of the action with the value of the Hamiltonian in the setting of Lagrangian Floer homology.
We expect that the main result of this section is well-known to experts, although we could not locate such a statement in the literature.
\Cref{sec:derivative} combines the results of \Cref{sec:genericity} and \Cref{sec:action} in order to establish the Derivative Property.

%%%
%%%
%%%

\section{Proofs of Theorems \ref{thm:mainthm} and \ref{thm:otherthm}}
\label{sec:thms}
We prove Theorems \ref{thm:mainthm} and \ref{thm:otherthm} in this Section.
We begin with the proof of \Cref{thm:otherthm}, which reviews our framework and sets up the related proof of \Cref{thm:mainthm}.
We then state a new property of the spectral invariants that we need to prove \Cref{thm:mainthm}, which gets proven over the course of the later sections.

Fix a smooth Jordan curve $\gamma \subset \bC$ and an angle $0 < \theta < \pi$.
Jordan Floer homology is the homology of an integer-graded $\bF_2$-chain complex $\JFC(\gamma,\theta)$, which is generically generated by inscriptions of a $\theta$-rectangle in $\gamma$.
It carries a real-valued action filtration $\cA : \JFC(\gamma,\theta) \to \bR$.
Its homology is a 2-dimensional graded $\bF_2$-vector space $\JF(\gamma,\theta) \approx (\bF_2)_{(2)} \oplus (\bF_2)_{(1)}$, and the action filtration descends to a spectral invariant $\ell : \JF(\gamma,\theta) \to \bR$.
We abbreviate by $\ell(\gamma,\theta)$ the spectral invariant attached to the homogeneous element in grading 2.
The spectral invariant enjoys several useful properties for a fixed smooth curve $\gamma$ \cite{greenelobb3,leczap}:
\begin{itemize}
\item {\bf Monotonicity:}
it defines a monotone increasing bijection $\ell(\gamma,\cdot) : [0,\pi] \longrightarrow [0,A(\gamma)]$.
\item {\bf Lipschitz Continuity:}
it is Lipschitz continuous in $\theta$ with Lipschitz constant $R(\gamma)^2$.
\item {\bf Spectrality:}
there exists an inscribed $\theta$-rectangle $P$ in $\gamma$ for which $\ell(\gamma,\theta) = \cA(P)$.
\item {\bf Triangle Inequality:}
for all angles $\theta_1, \theta_2 \ge 0$ such that $\theta_1 + \theta_2 \le \pi$, we have
\[
\ell(\gamma,\theta_1+\theta_2) \le \ell(\gamma,\theta_1) + \ell(\gamma,\theta_2).
\]
\end{itemize}

Now drop the smoothness hypothesis, so that $\gamma \subset \bC$ denotes an arbitrary Jordan curve.
Approximate $\gamma$ in $C^0$ by a sequence of smooth Jordan curves $\gamma_n$, each enclosing area $A(\gamma)$.
By Spectrality, there exists a $\theta$-rectangle $P_n$ inscribed in $\gamma_n$ such that $\cA(P_n) = \ell(\gamma_n,\theta)$.
By passing to a convergent subsequence, we may assume that (a) $L := \lim_{n \to \infty} \ell(\gamma_n,\theta)$ exists and (b) the $P_n$ converge to a point set $P$ inscribed in $\gamma$.
A priori, $P$ is either a single point or a $\theta$-rectangle.

\begin{lem}[Lemma 5.1, \cite{greenelobb3}]
\label{lem:rect}
If $\gamma$ is rectifiable and $0 < L < A(\gamma)$, then $P$ is an inscribed $\theta$-rectangle in $\gamma$. \hfill $\qed$\end{lem}

\begin{proof}[Proof of \Cref{thm:otherthm}]
Fix a value $0 < \theta < B(\gamma)$, and define $\gamma_n$ and $L$ as above.
We will verify the hypothesis of Lemma \ref{lem:rect}.

By an iterated application of the Triangle Inequality, we have
\[
\ell(\gamma_n, \theta_1+\cdots+\theta_k) \le \ell(\gamma_n,\theta_1) + \cdots + \ell(\gamma_n,\theta_k)
\]
for all $n \ge 1$, for all $k \ge 1$, and for all angles $\theta_1,\dots,\theta_k \ge 0$ such that $\theta_1+ \cdots + \theta_k \le \pi$.
Take $\theta_1 = \cdots = \theta_k = \pi/k$ and use $\ell(\gamma_n,\pi) = A(\gamma)$ to obtain $\ell(\gamma_n,\pi/k) \ge A(\gamma) / k$ for all $k \ge 1$.
Fix $k$ large enough so that $\theta \ge \pi/k$.
By Monotonicity, it follows that $\ell(\gamma_n,\theta) \ge \ell(\gamma_n,\pi/k) \ge A(\gamma)/k$ for all $n$.
Consequently, $L > 0$.

On the other hand, Lipschitz Continuity and $\ell(\gamma_n,0) = 0$ gives $\ell(\gamma_n,\theta) \le R(\gamma_n)^2 \cdot \theta$ for all $n$.
Hence $L \le \lim_{n \to \infty} R(\gamma_n)^2 \cdot \theta < \lim_{n \to \infty} R(\gamma_n)^2 \cdot B(\gamma) = A(\gamma)$.

We have shown that $0 < L < A(\gamma)$, so $\gamma$ inscribes a $\theta$-rectangle, by Lemma \ref{lem:rect}.
\end{proof}

For the proof of \Cref{thm:mainthm}, we record one more property of the spectral invariants:

\begin{itemize}
\item {\bf Derivative Property:}
for an open dense subset of smooth Jordan curves $\gamma$, there exists a set $\Theta(\gamma) \subset (0,\pi)$ of full measure such that every $\theta \in \Theta(\gamma)$ is a point of differentiability of $\ell(\gamma,\cdot)$, and there exists an inscribed $\theta$-rectangle $P$ in $\gamma$ such that $\ell'(\gamma,\theta) = R(P)^2$.
\end{itemize}

\noindent
Here $R(P)$ denotes the radius (half of the diameter) of $P$.
The Derivative Property is stated more precisely below in \Cref{thm:deriv1}.

\begin{proof}[Proof of \Cref{thm:mainthm}]
Suppose that $\gamma \subset \bC$ is an arbitrary Jordan curve, and approximate it in $C^0$ by smooth Jordan curves $\gamma_n$ as before.
We may assume that each $\gamma_n$ belongs to the open dense set of Jordan curves to which the Derivative Property applies.

Fix a value $\epsilon > 0$, and let $S(\epsilon,\gamma_n) \subset \Theta(\gamma_n)$ denote the (measurable) subset of angles $\theta$ such that $\ell'(\gamma_n,\theta) \ge \epsilon$.
Because $\Theta(\gamma_n)$ is dense in $(0,\pi)$, it follows that at almost every point $\theta$ of the complement $S(\epsilon,\gamma_n)^c \subset (0,\pi)$, the derivative $\ell'(\gamma_n,\theta)$ is defined and satisfies $\ell'(\gamma_n,\theta) < \epsilon$.

Lipschitz Continuity implies that the function $\ell(\gamma_n,\theta)$ is absolutely continuous, so its derivative is defined almost everywhere and obeys the Fundamental Theorem of Calculus:
\[
A(\gamma_n) = \ell(\gamma_n,\pi) - \ell(\gamma_n,0) = \int_0^\pi \ell'(\gamma_n,\theta) d \theta.
\]
We break the integral into the integral over $S(\epsilon,\gamma_n)$ and its complement $S(\epsilon,\gamma_n)^c$.
We bound $\ell'(\gamma_n,\theta)$ from above by $R(\gamma_n)^2$ on $S(\epsilon,\gamma_n)$  using Lipschitz Continuity.
We bound $\ell'(\gamma_n,\theta)$ from above by $\epsilon$ almost everywhere on $S(\epsilon,\gamma_n)^c$, and we bound the measure of $S(\epsilon,\gamma_n)^c$ from above by $\pi$.
Thus,
\[
A(\gamma_n) \le R(\gamma_n)^2 \cdot \mu(S(\epsilon,\gamma_n)) + \epsilon \cdot \pi.
\]
Rearranging gives $\mu(S(\epsilon,\gamma_n)) \ge (A(\gamma_n) - \pi \epsilon)/R(\gamma_n)^2$.

Let $\cS(\epsilon,\gamma) \subset (0,\pi)$ consist of those angles contained in infinitely many of the sets $S(\epsilon,\gamma_n)$.
By the elementary measure theoretic Lemma \ref{lem:measure} below, it has measure $\mu(\cS(\epsilon,\gamma)) \ge (A(\gamma) - \pi \epsilon)/R(\gamma)^2$.

Select any angle $\theta \in \cS(\epsilon,\gamma)$.
We claim that there is a $\theta$-rectangle $P$ inscribed in $\gamma$ with $R(P)^2 \ge \epsilon$.
By construction, there exists an infinite sequence $n_k = n_k(\theta)$ such that $\ell'(\gamma_{n_k},\theta) \ge \epsilon$ for all $k \ge 1$.
By the Derivative Property, there exists an inscribed $\theta$-rectangle $P_k$ in $\gamma_{n_k}$ for all $k \ge 1$, each satisfying $R(P_k)^2 \ge \epsilon$.
By compactness, the $P_k$ possess a convergent subsequence contained in $\gamma$.
Its limit is the desired $\theta$-rectangle $P$.

Lastly, define $\cS(\gamma) = \bigcup_{\epsilon > 0} \cS(\epsilon,\gamma)$.
Then $\mu(\cS(\gamma)) \ge B(\gamma)$, and $\gamma$ inscribes a $\theta$-rectangle, for all $\theta \in \cS(\gamma)$.
\end{proof}

The proof of \Cref{thm:mainthm} establishes the following result:

\begin{cor}
If $\gamma \subset \bC$ is a Jordan curve and $\epsilon > 0$, then there exists a subset $\cS(\epsilon,\gamma) \subset (0, \pi)$ of measure at least $(A(\gamma) - \pi \epsilon) / R(\gamma)^2$ such that for all $\theta \in \cS(\epsilon,\gamma)$, $\gamma$ inscribes a $\theta$-rectangle $P$ with $R(P)^2 \ge \epsilon$. \hfill \qed
\end{cor}

\begin{scholium}
If $\gamma \subset \bC$ is a Jordan curve and $k$ is a positive integer, then $\mu(\cS(\gamma) \cap [0,\pi/k]) \ge B(\gamma)/k$.
\end{scholium}

\begin{proof}
Repeat the proof of \Cref{thm:mainthm}, working over the interval $[0,\pi/k]$ and using the inequality $\ell(\gamma_n,\pi/k) \ge A(\gamma_n)/k$ established in the course of the proof of \Cref{thm:otherthm}.
\end{proof}

Lastly, we state and prove the elementary measure theoretic result used in the proof of \Cref{thm:mainthm}:

\begin{lem}
	\label{lem:measure}
	Suppose that $(X,\mu)$ is a finite measure space, $S_1,S_2,\dots$ is a sequence of measurable subsets of $X$, and $\liminf_{n \to \infty} \mu(S_n) = \mu_0$. 
	Let $S \subseteq X$ be the set of points contained in infinitely many $S_n$.
	Then $S$ is measurable, and $\mu(S) \ge \mu_0$.
\end{lem}

\begin{proof}
	For each $k \geq 1$, let $T_k$ denote the set of points contained in at least $k$ of the subsets $S_n$.
	
	First we show that $T_k$ is measurable, for all $k$.
	Its complement $T_k^c$ is the set of points contained in at most $k-1$ of the subsets $S_n$.
	Thus $T_k^c$ is the countable union, over all subsets $U \subset \bZ^+$ with $|U| \le k-1$, of $\bigcap_{i \in U} S_i \setminus \bigcup_{j \in U^c} S_j$.
	It follows that $T_k$ is in the $\sigma$-subalgebra generated by the $S_n$, so it is measurable.
	
	Next we show that $\mu(T_k) \ge \mu_0$, for all $k$.
	A point in $X$ belongs to at most $k-1$ of the sets $S_n \cap T_k^c$, $n \ge 1$.
	It follows that $\sum_{n=1}^\infty \mu(S_n \cap T_k^c) \le (k-1) \cdot \mu(X) < \infty$.
	Hence $\lim_{n \to \infty} \mu(S_n \cap T_k^c) = 0$, so $\liminf_{n \to \infty} \mu(S_n \cap T_k) = \liminf_{n \to \infty} \mu(S_n) = \mu_0$, so $\mu(T_k) \ge \mu_0$.
	
	Finally, we have $T_1 \supseteq T_2 \supseteq \cdots$ and $S = \bigcap_{k \ge 1} T_k$.
	We conclude that $S$ is measurable, and $\mu(S) \ge \mu_0$, as desired.
\end{proof}

%%%
%%%
%%%

\section{Genericity}
\label{sec:genericity}

The purpose of this section is to make precise the result that for an open dense subset of smooth Jordan curves, the space of inscribed rectangles and binormals has the structure of a smooth, closed 1-dimensional manifold.
Its proof is given in the Appendix.
It is an analogue in the smooth setting of a result of Schwartz for PL curves \cite[Theorem 1.5]{schwartz2020}.
It is this type of Jordan curve to which the Derivative Property applies.
We begin by framing its context.

Equip $\bC^2$ with the standard symplectic form $\omega$.
Define the Hamiltonian function
\begin{eqnarray*}
H : \bC^2 &\longrightarrow& \bR, \\
(z,w) &\longmapsto& \frac14|z-w|^2.
\end{eqnarray*}
Let $X_H$ denote the Hamiltonian vector field, implicitly defined via $\omega(X_H,-)=dH$.
Its flow defines a Hamiltonian circle action.
For $\theta \in S^1 = [0,2\pi]/(0 \sim 2 \pi)$, its time-$\theta$ flow $\phi^H_\theta = \rho_\theta : \bC^2 \to \bC^2$ rotates through angle $\theta$ about the diagonal line $\Delta(\bC) \subset \bC^2$.
In even more concrete terms, given $(z,w) \in \bC^2$, rotate the line segment $\overline{zw}$ clockwise through angle $\theta$ around its midpoint; then $\rho_\theta(z,w) = (z',w')$ mark the endpoints of the rotated line segment.
Thus, if $\rho_\theta(z,w) = (z',w')$, then either $z=z'=w=w'$, or else $z,z',w,w'$ mark the vertices of a $\theta$-rectangle in clockwise order.

Assume that $\gamma \subset \bC$ is a smooth Jordan curve and $\theta \in S^1$.
Then $L_0 = \gamma \times \gamma$ and $L_1 := \rho_\theta(\gamma \times \gamma)$ are Lagrangian tori in $(\bC^2,\omega)$.
If $\theta \ne 0,\pi$ and $(z',w') = \rho_\theta(z,w)$ is a point of intersection of $L_0$ and $L_1$, then either $(z',w')$ belongs to the diagonal loop $\Delta(\gamma) = \{ (z,z) : z \in \gamma\}$, or else $(z,z',w,w')$ are the vertices of an inscribed $\theta$-rectangle in $\gamma$.
The construction of the Jordan Floer chain complex is a version of Lagrangian Floer homology adapted to the pair $(L_0,L_1)$, using intersection points and holomorphic strips which avoid the line $\Delta(\bC)$.

We parametrize all of the inscribed rectangles in $\gamma$ at once by forming the intersection of the pair of subspaces
\[
\cL_0 = L_0 \times S^1, \quad \cL_1 = \bigcup_{\theta \in S^1 \setminus \{0,\pi\}} \rho_\theta(L_0) \times \{ \theta \}
\]
of $\bC^2 \times S^1$.
The intersection $\cL_0 \cap \cL_1$ admits a natural compactification, as follows.
A {\em binormal} of $\gamma$ is an ordered pair $(z,w) \in L_0 \setminus \Delta(\gamma)$ such that $T_z \gamma$ and $T_w \gamma$ are orthogonal to the line segment $\overline{zw}$.
Informally, it corresponds to an inscribed $\theta$-rectangle in $\gamma$ in the limiting case $\theta \in \{0,\pi\}$.
As we argue below, if we restrict $H$ to $L_0$, then the critical points of $H | L_0$ consist of binormals of $\gamma$ and the loop $\Delta(\gamma)$ (Lemma \ref{lem:binormal}; we already observed this in \cite{greenelobb3}).
Moreover, $\Crit(H | L_0) \times \{0,\pi\}$ compactifies $\cL_0 \cap \cL_1$, and one component of the compactification is $\Delta(\gamma) \times S^1$ (Lemma \ref{lem:smooth2}).
Define
\[
\cR(\gamma) = \left( \cL_0 \cap \cL_1 \right) \,  \bigsqcup \, \left( \Crit(H | L_0) \times \{ 0 , \pi \} \right) \setminus \Delta(\gamma) \times S^1.
\]
Thus, $\cR(\gamma)$ parametrizes the inscribed rectangles and binormals in $\gamma$.

Finally, let $C^\infty(S^1,\bC)$ denote the space of smooth functions from $S^1$ to $\bC$, equipped with the Whitney $C^\infty$ topology.
Within it sits the open subspace $\Emb^\infty(S^1,\bC)$ of smooth embeddings, i.e. injective maps with nonvanishing derivative.
It is the space of parametrized smooth Jordan curves.

\begin{thm}[{\bf Genericity Theorem}]
\label{thm:genericity}
With respect to the $C^\infty$ topology on $\Emb^\infty(S^1,\bC)$, there exists an open dense subset $\Emb^\infty_\mathrm{reg}(S^1,\bC)$ such that if $f \in \Emb^\infty_\mathrm{reg}(S^1,\bC)$ and $\gamma = f(S^1)$, then the subspace $\cR(\gamma) \subset \gamma \times \gamma \times S^1$ is a smooth, closed 1-dimensional submanifold.
\end{thm}

%%%
%%%
%%%

\section{The Derivative of the Action}
\label{sec:action}

Suppose that
\begin{itemize}
\item
$(M,\omega)$ is a symplectic manifold,
\item
$L \subset M$ is a Lagrangian submanifold,
\item
$H : M \longrightarrow \bR$ is a Hamiltonian,
\item
the time-$s$ flow $\phi^H_s$ of the Hamiltonian vector field $X_H$ is defined for a value $s \in \bR$, and
\item
$p_s \in L \cap \phi^H_s(L)$ is a point of intersection.
\end{itemize}
Under favorable hypotheses, this point has a well-defined symplectic action $\cA(p_s) \in \bR$.
If $s \in I$ varies in an open interval and $p_s \in L \cap \phi^H_s(L)$ varies smoothly in $s$, then we can study the smooth function $I \longrightarrow \bR$, $s \longmapsto \cA(p_s)$.
The result of this section equates the derivative of $\cA(p_s)$ at time $s=T$ with the value of $H(p_T)$ (Proposition \ref{prop:deriv}).
We suspect this result is well-known, although we could not locate it in the literature.
We proceed to fill in the required hypotheses before formally stating the result.

Write $q_s \in L$ for the point such that $p_s = \phi^H_s(q_s)$.
There is a corresponding trajectory $\tau_s : [0,s] \to M$ of $\phi^H_s$, i.e.~$\tau_s(0) = q_s, \tau_s(1) = p_s$, and $(d/dt) \tau_s(t) = X_H \circ \tau_s(t)$, for all $0 \le t \le s$.
Suppose that $\pi_1(M,L) = \{ 1 \}$.
Then there exists a {\em capping} of $\tau_s$, i.e.~a homotopy rel endpoints $\widehat{\tau}_s : [0,s] \times [0,1] \to M$ from $\widehat{\tau}_s(t,0) = \tau_s(t)$ to a path $\widehat{\tau}_s(t,1)$ whose image is contained in $L$.
The homotopy class of $\widehat{\tau}_s$ is unique if $\pi_2(M,L) = 0$, as well as in the setting of Jordan Floer homology, in which we impose the additional constraint that $\widehat{\tau}_s$ is disjoint from the divisor $\Delta(\bC)$ \cite[Lemma 2.5]{greenelobb3}.
Assuming this is the case, the symplectic action of $p_s$ is the well-defined real number given by
	\begin{align*}
		\cA(p_s) &= \int_0^s H(\tau_s(t)) dt - \int_{[0,s]\times[0,1]} \widehat{\tau}_s^* \omega.
	\end{align*}
Now suppose that $I \subset \bR$ is an open interval, $\phi^H_s : M \longrightarrow M$ is defined for all $s \in I$, and
\[
p : I \longrightarrow M, \quad s \longmapsto p_s \in \phi^H_s(L) \cap L,
\]
is a smooth path of intersection points.

\begin{prop}
	\label{prop:deriv}
	Under the preceding hypotheses, the derivative of $\cA(p_s)$ at $s=T$ equals $H(p_T)$.
\end{prop}

As remarked, we could not locate a proof in the literature, so for certainty, we give two.

\begin{proof}[Proof 1.]
	Write $D$ for the polygonal region of the $(s,t)$ plane with vertices
	\[ (T,0), (T + \delta,0), (T+\delta,T + \delta), (T,T) {\rm .} \]
	Define $\tau : D \longrightarrow M$ by $(s,t) \longmapsto \tau_s(t)$.
	We obtain a capping of $\tau_{T+\delta}$ by concatenating a capping of $\tau_T$ with $\tau$.
	Then we have that
	\begin{align*}
	\cA(p_{T + \delta}) - \cA(p_T) &= \int_0^{T+ \delta} H(\tau_{T + \delta}(t)) dt - \int_0^{T} H(\tau_{T}(t)) dt - \int_{\tau(D)} \omega  \\
	&= \int_0^{T+ \delta} H(\tau_{T + \delta}(t)) dt - \int_0^{T} H(\tau_{T}(t)) dt - \int_D \omega(X_H, \partial_s \tau_s(t)) dsdt  \\
	&= \int_0^{T+ \delta} H(\tau_{T + \delta}(t)) dt - \int_0^{T} H(\tau_{T}(t)) dt - \int_D \partial_s H(\tau_s(t)) dsdt  \\
	&= \int_{T}^{T + \delta}  H(\tau_t(t)) dt {\rm .}
	\end{align*}
	Dividing through by $\delta$ and taking $\delta \rightarrow 0$ gives
	\[ \left. \frac{d}{ds} \cA(p_s)\right|_{s=T} = H(\tau_T(T)) = H(p_T), \]
	as required.
\end{proof}

\begin{proof}[Proof 2.]
For any smooth path $x : [0,1] \to M$ with $x(0), x(1) \in L$, we can find a capping $\widehat{x} : [0,1] \times [0,1] \to M$ and define the symplectic action of $x$ with respect to the Hamiltonian $s H$ by the formula
\[
\cA(x,sH) = \int_0^1 sH \circ x(t) \, dt - \int_{[0,1]\times[0,1]} x^* \omega.
\]
As before, it is independent of $\widehat{x}$ under the operating assumptions (provided $x$ is disjoint from $\Delta(\bC)$, in the setting of Jordan Floer homology).
A trajectory $\tau_s : [0,s] \to M$ of $\phi^H_s$ corresponds to a trajectory $\sigma_s :[0,1] \to M$ of $\phi^{s H}_1$ by a linear reparametrization: $\sigma_s(t) = \tau_s(t/s)$.
The trajectory $\sigma_T$ is a critical point of $\cA(x,TH)$.
In particular, taking the smooth family of paths $\sigma_s$ passing through $\sigma_T$, we have
\begin{equation}
\label{e:criticalpoint}
\left. \frac{d}{ds} \cA(\sigma_s,TH) \right|_{s=T} = 0.
\end{equation}
On the other hand,
\[
\cA(\sigma_s,sH) = \cA(\sigma_s,TH) + (s-T) \int_0^1 H \circ \sigma_s(t) dt.
\]
The left side matches the action $\cA(p_s)$ in the earlier notation, and the integrand matches $H(p_s)$.
Hence 
\[
\cA(p_s) = \cA(\sigma_s,TH) + (s-T) H(p_s).
\]
Taking the derivative at $s = T$, using \eqref{e:criticalpoint} for the first term and the product rule for the second, returns the desired result.
\end{proof}

%%%
%%%
%%%

\section{The Derivative Property}
\label{sec:derivative}

This section combines the Genericity Theorem for smooth Jordan curves of \Cref{sec:genericity} and the derivative property for the action of \Cref{sec:action} to prove the Derivative Property of the spectral invariants.
For a smooth Jordan curve $\gamma \subset \bC$, recall the space of inscribed rectangles and binormals $\cR(\gamma)$.
Recall as well \Cref{thm:genericity}, which asserts that the space $\Emb^\infty_\mathrm{reg}(S^1,\bC)$ of parametrized Jordan curves for which $\cR(\gamma)$ is a smooth manifold forms an open dense subset of $\Emb^\infty(S^1,\bC)$.

\begin{thm}[{\bf Derivative Property}]
\label{thm:deriv1}
Let $f \in \Emb^\infty_\mathrm{reg}(S^1,\bC)$ and let $\gamma = f(S^1) \subset \bC$.
Then there exists a subset $\Theta(\gamma) \subset (0,\pi)$ of full measure such that every $\theta \in \Theta(\gamma)$ is a point of differentiability of $\ell(\gamma,\cdot)$, and there exists an inscribed $\theta$-rectangle $P$ in $\gamma$ such that $\ell'(\gamma,\theta) = R(P)^2$.
\end{thm}

\begin{proof}
Consider the smooth map $p : \cR(\gamma) \longrightarrow S^1$ obtained by restricting the projection map $\gamma \times \gamma \times S^1 \longrightarrow S^1$ to $\cR(\gamma)$.
Let $\mathrm{Reg}(p) \subset S^1$ denote its set of regular values.
Because $\cR(\gamma)$ is smooth and compact, Sard's theorem guarantees that $\mathrm{Reg}(p)$ is open and dense in $S^1$.
Let $D(\gamma) \subset (0,\pi)$ denote the set of points at which $\ell(\gamma,\cdot)$ is differentiable.
By Lipschitz Continuity, $D(\gamma)$ has full measure.
Let $\Theta(\gamma) = \mathrm{Reg}(p) \cap D(\gamma)$.
It follows that $\Theta(\gamma) \subset (0,\pi)$ has full measure.

Fix $\theta \in \Theta(\gamma)$, and let $I$ denote the component of $\mathrm{Reg}(p)$ containing $\theta$.
Because $\mathrm{Reg}(p)$ is open, $I$ is an open interval, and $p^{-1}(I)$ consists of finitely many components $I_1,\dots,I_n \subset R(\gamma)$, each mapped diffeomorphically to $I$ by $p$.
For each angle $\psi \in I$ and index $k=1,\dots,n$, let $(z_k(\psi),w_k(\psi),\psi)$ denote the preimage of $\psi$ in $I_k$.
The restriction of $\cA$ to $I_k$ is a smooth, real-valued function, and its $\psi$-derivative at a point $(z,w,\psi) \in I_k$ is equal to $H(z,w) = \frac14|z-w|^2$ by Proposition \ref{prop:deriv}.

Pick a sequence of values $\theta_1, \theta_2, \dots \in I \setminus \{\theta\}$ limiting to $\theta$.
By Spectrality, for each $\theta_i$, there exists an index $k(i) \in \{1,\dots,n\}$ such that $\ell(\gamma,\theta_i) = \cA(z_{k(i)}(\theta_i),w_{k(i)}(\theta_i),\theta_i)$.
By passing to a subsequence, we may assume that there exists a single value $k \in \{1,\dots,n\}$ such that $k = k(i)$ for all $i$.
By continuity of both $\ell$ and $\cA$, we have
\[
\ell(\gamma,\theta) = \lim_{i \to \infty} \ell(\gamma,\theta_i) = \lim_{i \to \infty} \cA(z_k(\theta_i),w_k(\theta_i),\theta_i) = \cA(z_k(\theta),w_k(\theta),\theta).
\]
Moreover, because $\theta \in D(\gamma)$ is a point of differentiability of $\ell(\gamma,\cdot)$, we can compute
\begin{eqnarray*}
\ell'(\gamma,\theta) &=& \lim_{i \to \infty} \frac{\ell(\gamma,\theta_i) - \ell(\gamma,\theta)}{\theta_i-\theta} \\
&=&  \lim_{i \to \infty} \frac{\cA(z_k(\theta_i),w_k(\theta_i),\theta_i) - \cA(z_k(\theta),w_k(\theta),\theta)}{\theta_i-\theta} \\
&=& \frac{d}{d \theta} \cA(z_k(\theta),w_k(\theta),\theta) \\
&=& H(z_k(\theta),w_k(\theta)) \\
&=& R(P)^2,
\end{eqnarray*}
where $P$ is the $\theta$-rectangle inscribed in $\gamma$ with diagonals $\overline{z_k(\theta)w_k(\theta)}$ and $R_\theta(\overline{z_k(\theta)w_k(\theta)})$.
\end{proof}

%%%
%%%
%%%

\appendix

\section{Proof of the Genericity Theorem}

The proof of \Cref{thm:genericity} proceeds in two stages.
First, we study the space $\Emb^\infty_1$ of smooth embeddings for which $\cL_0$ and $\cL_1$ intersect transversely away from $\Delta(\gamma) \times S^1$.
These have the property that $\cR(\gamma)$ is a smooth manifold away from $\gamma \times \gamma \times \{0,\pi\}$.
Second, we study the space $\Emb^\infty_2$ of smooth embeddings for which $H | L_0$ is a Morse function away from $\Delta(\gamma)$.
These have the property that $\cR(\gamma)$ is a smooth manifold in a neighborhood of $\gamma \times \gamma \times \{0,\pi\}$.
We show that both $\Emb^\infty_1$ and $\Emb^\infty_2$ are open, dense subspaces of $\Emb^\infty(S^1,\bC)$, hence so is $\Emb^\infty_\mathrm{reg} = \Emb^\infty_1 \cap \Emb^\infty_2$.

\subsection{Rectangles.}
\label{sec:rectangles}

For the first step, we use a related picture for parametrizing inscribed rectangles in a Jordan curve in the spirit of Schnirelmann's work.
Given $\theta \in S^1 \setminus \{0,\pi\}$, form
\[
\graph(\rho_\theta) = \{ ((z,w), \rho_\theta(z,w)) : (z,w) \in \bC^2 \} \subset \bC^2 \times \bC^2.
\]
It is the space of quadruples $(z,w,z',w') \in \bC^4$ such that either $z=w=z'=w'$, or else $z,z',w,w'$ label the vertices of a $\theta$-rectangle in clockwise order, so that $\overline{zw}$ makes an angle $\theta$ with $\overline{z'w'}$ at the rectangle's center.
Thus, $\graph(\rho_\theta) \subset \bC^4$ is the space of (cyclically ordered, possibly degenerate) $\theta$-rectangles in $\bC$.
Form the pair of subspaces
\[
\gamma^4 \times S^1, \quad R := \bigcup_{\theta \in S^1 \setminus \{0,\pi\}} \graph(\rho_\theta) \times \{\theta\}
\]
of $\bC^4 \times S^1$.
Thus, $R$ is the space of all (cyclically ordered, possibly degenerate) rectangles in $\bC$, resolved by the aspect angle $\theta \ne 0,\pi$.
It is clear from this description that $R$ is a smooth submanifold.
The intersection $(\gamma^4 \times S^1) \cap R$ therefore parametrizes (cyclically ordered, possibly degenerate) inscribed rectangles in $\gamma$.
Moreover,
\[
(\rho_\theta(z,w),\theta) \in \cL_0 \cap \cL_1 \iff ((z,w),\rho_\theta(z,w),\theta) \in (\gamma^4 \times S^1) \cap R.
\]
We leave the proof of the following result as an exercise:

\begin{lem}
The subspaces
$\cL_0$ and $\cL_1$ intersect transversely at $(\rho_\theta(z,w),\theta)$ if and only if $\gamma^4 \times S^1$ and $R$ intersect transversely at $((z,w),\rho_\theta(z,w),\theta)$. \qed
\end{lem}

Let
\[
\Delta^4(\gamma) = \{ (z,z,z,z) \in \bC^4 : z \in \gamma \}.
\]
In the parametrization of Section \ref{sec:genericity}, the degenerate inscriptions of rectangles in $\gamma$ are the points of $\Delta(\gamma) \times (S^1 \setminus \{0,\pi\})$, while in the new one, they are the points of $\Delta^4(\gamma) \times (S^1 \setminus \{0,\pi\})$.
These loci are controlled, in the following sense:

\begin{lem}
\label{lem:cleanlocus}
For every differentiable Jordan curve $\gamma \subset \bC$,
\begin{enumerate}
\item
$\cL_0$ and $\cL_1$ intersect cleanly along $\Delta(\gamma) \times (S^1 \setminus \{0,\pi\})$, and
\item
$\gamma^4 \times S^1$ and $R$ intersect cleanly along $\Delta^4(\gamma) \times (S^1 \setminus \{0,\pi\})$.
\end{enumerate}
\end{lem}

\noindent
Recall that if $A$, $B$, and $C$ are differentiable submanifolds of an ambient manifold $M$, then $A$ and $B$ intersect {\em cleanly along} $C$ if $C \subset A \cap B$ and $T_x C = T_x A \cap T_x B$ for all $x \in C$.
Moreover, $A$ and $B$ intersect {\em cleanly} without qualification if $A$ and $B$ intersect cleanly along each component of $A \cap B$.

\begin{proof}
We prove (1); the proof of (2) is similar.
Suppose that $(z,z,\theta) \in \Delta(\gamma) \times (S^1 \setminus \{0,\pi\})$.
We compute
\[
T_{(z,z,\theta)} \cL_0 = T_z \gamma \oplus T_z \gamma \oplus T_\theta (S^1 \setminus \{0,\pi\})
\]
and
\[
T_{(z,z,\theta)} \cL_1 = d \rho_\theta( T_z \gamma \oplus T_z \gamma) \oplus T_\theta (S^1 \setminus \{0,\pi\}) = \rho_\theta(T_z \gamma \oplus T_z \gamma) \oplus T_\theta (S^1 \setminus \{0,\pi\}).
\]
Here we use the trivialization $T \bC^2 = \bC^2 \oplus \bC^2$ and the fact that $\rho_\theta$ is linear.
For $\theta \ne 0, \pi$, the subspaces $T_z \gamma \oplus T_z \gamma$ and $\rho_\theta(T_z \gamma \oplus T_z \gamma)$ intersect in the diagonal embedding of $T_z \gamma$ in $T_{(z,z)} \bC^2$.
Therefore, the two tangent spaces intersect in $\Delta(T_z \gamma) \oplus T_\theta (S^1 \setminus \{0,\pi\}) = T_{(z,z,\theta)} \Delta(\gamma) \times (S^1 \setminus \{0,\pi\})$.
This establishes the cleanliness of the intersection along $\Delta(\gamma) \times (S^1 \setminus \{0,\pi\})$.
\end{proof}

\begin{cor}
\label{cor:niceintersection}
The following statements are equivalent:
\begin{itemize}
\item
$\cL_0$ and $\cL_1$ intersect cleanly, and transversely away from $\Delta(\gamma) \times (S^1 \setminus \{0,\pi\})$; and
\item
$\gamma^4 \times S^1$ and $R$ intersect cleanly, and transversely away from $\Delta^4(\gamma) \times (S^1 \setminus \{0,\pi\})$.  \qed
\end{itemize}
\end{cor}

With the new parametrization in hand, we turn to the first piece of the Genericity Theorem.
Its proof involves the use of the Sard-Smale theorem for Banach manifolds.
Since $C^\infty(S^1,\bC)$ is not itself a Banach manifold, this approach requires passage through maps with possibly less regularity, followed by a simple topological argument.

For every $r \in \bZ^+$, let $C^r(S^1,\bC)$ denote the space of $r$-times continuously differentiable functions from $S^1$ to $\bC$.
It admits a metric, the $C^r$ metric, with respect to which it is a Banach manifold.
The topology induced by the metric is the $C^r$ topology.
The space of smooth functions $C^\infty(S^1,\bC)$ is a subspace of $C^r(S^1,\bC)$.
It inherits the $C^r$ topology.
By a basic result, $C^\infty(S^1,\bC)$ is dense in $C^r(S^1,\bC)$ in the $C^r$ topology.
The $C^\infty$ topology on $C^\infty(S^1,\bC)$ has as a basis the union of the $C^r$ topologies over all finite $r$.
It is also induced by a metric, the $C^\infty$ metric.
However, $C^\infty(S^1,\bC)$ is only a Frech\'et manifold with respect to this metric, not a Banach manifold.

Let $\Emb^r = \Emb^r(S^1,\bC) \subset C^r(S^1,\bC)$ denote the subspace of $r$-times continuously differentiable embeddings.
It is an open subspace of $C^r(S^1,\bC)$.
Likewise, let $\Emb^\infty = \Emb^\infty(S^1,\bC)$ denote the space of smooth embeddings.
Thus, $\Emb^\infty = \Emb^r \cap C^\infty(S^1,\bC)$.
Because $\Emb^r$ is open, it follows that $\Emb^\infty$ is $C^r$-dense in $\Emb^r$.
For any $f \in \Emb^r$, we write $\gamma = f(S^1)$.

For all $r \in \bZ^+$, define
\[
\Emb^r_1 := \{ f \in \Emb^r : \gamma^4 \times S^1 \textup{ intersects } R \textup{ cleanly, and transversely away from } \Delta^4(\gamma) \times (S^1 \setminus \{0,\pi\}) \}.
\]
For any set $X$, let $X^{(n)}$ denote the $n$-tuples of distinct points of $X$.
Let
\[
R^* = R \cap ( \bC^{(4)} \times ( S^1 \setminus \{0,\pi\})).
\]
Note that
\[
(\gamma^4 \times S^1) \cap R \setminus (\Delta^4(\gamma) \times (S^1 \setminus \{0,\pi\})) = (\gamma^{(4)} \times (S^1 \setminus \{0,\pi\})) \cap R^*:
\]
both sides parametrize the (nondegenerate) inscribed rectangles in $\gamma$.
By Lemma \ref{lem:cleanlocus}, we therefore have
\[
\Emb^r_1 := \{ f \in \Emb^r : \gamma^{(4)} \times (S^1 \setminus \{0,\pi\})  \textup{ intersects } R^* \textup{ transversely} \}.
\]
Lastly, define
\[
\Emb^\infty_1 := \Emb^\infty \cap \Emb^r_1,
\]
and note that it is independent of the choice of $r$.

\begin{lem}
\label{lem:transverse1}
The space $\Emb^r_1$ is open and dense in $\Emb^r$ in the $C^r$ topology, for all $r \in \bZ^+$.
\end{lem}

The proof of Lemma \ref{lem:transverse1} follows a standard line of argument.

\begin{proof}[Proof of Lemma \ref{lem:transverse1}]
The set $\Emb^r_1$ is open in the $C^r$ topology, because the transversality condition defining it is an open condition.
We argue that $\Emb^r_1$ is $C^r$-dense in $\Emb^r$ by identifying it with the set of regular values of a Fredholm map between Banach manifolds and invoking the Sard-Smale theorem.

Form the universal space
\[
E := \Emb^r \times (S^1)^{(4)} \times (S^1 \setminus \{0,\pi\}),
\]
and the evaluation map 
\begin{eqnarray*}
\ev : E &\longrightarrow& \bC^{(4)} \times (S^1 \setminus \{0,\pi\}), \\
(f,z_1,z_2,z_3,z_4,\theta) &\longmapsto& (f(z_1),f(z_2),f(z_3),f(z_4),\theta).
\end{eqnarray*}
We claim that this map is a submersion.
To check this, select a point $(f,z_1,z_2,z_3,z_4,\theta)$ in the domain and a tangent vector $(v_1,v_2,v_3,v_4,w)$ to its image under $\ev$.
The tangent space to $\Emb^r$ at $f$ is identified with the space of $C^r$-vector fields along $f$.
Choose such a vector field $X$ which matches $v_k$ at $f(z_k)$ for each value $k=1,\dots,4$: this is possible to do, because the $z_k$ are pairwise distinct.
Then $d (\ev) (X,0,0,0,0,w) = (v_1,v_2,v_3,v_4,w)$.
This checks that $\ev$ is a submersion.

The subspace $R^* \subset \bC^{(4)} \times (S^1 \setminus \{0,\pi\})$ is a smooth submanifold of a smooth, finite-dimensional manifold.
The space $E$ is a Banach manifold.
It follows that
\[
\widetilde{R} := (\ev)^{-1}(R^*) \subset E
\]
is a Banach submanifold of finite codimension.
The projection map
\[
p : E \longrightarrow \Emb^r
\]
is Fredholm, since it is a submersion with finite dimensional fibers.
Because $\widetilde{R}$ has finite codimension within the domain, the restriction of $p$ to $\widetilde{R}$ is Fredholm, too.
By the Sard-Smale theorem, its set of regular values $\mathrm{Reg}(p | \widetilde{R})$ is $C^r$-dense in $\Emb^r$.

We claim that $\mathrm{Reg}(p | \widetilde{R}) = \Emb^r_1$.
Checking this is a matter of linear algebra and bookkeeping with tangent spaces.
We go through the argument in detail; the proof of the similar Proposition \ref{prop:E2} omits the similar details.

First, we show that (a) $f \in \mathrm{Reg}(p | \widetilde{R})$ if and only if (b) $p^{-1}(f) \pitchfork \widetilde{R}$.
Condition (a) holds if and only if, for all $x \in p^{-1}(f)$, the differential $d (p | \widetilde{R})_x$ surjects $T_x \widetilde{R}$ onto $T_f \Emb^r$.
This condition on $x$ is equivalent to the condition that $d p_x$ surjects $T_x \widetilde{R}$ onto $T_f \Emb^r$.
The map $d p_x$ surjects $T_x E$ onto $T_f \Emb^r$, and its kernel equals $T_x p^{-1}(f)$.
Hence $d p_x$ surjects $T_x \widetilde{R}$ onto $T_f \Emb^r$ if and only if $T_x p^{-1}(f) + T_x \widetilde{R} = T_x E^r$.
This holds for all $x \in p^{-1}(f)$ if and only if $p^{-1}(f) \pitchfork \widetilde{R}$, as required.

Second, we show that (b) $p^{-1}(f) \pitchfork \widetilde{R}$ if and only if (c) $f \in \Emb^r_1$.
Let $x \in p^{-1}(f) \cap \widetilde{R}$.
Then $x$ is a transverse point of intersection if and only if $T_x p^{-1}(f) + T_x \widetilde{R} = T_x E^r$.
Write $y = \ev(x)$.
As we argued above, $d (\ev)_x$ surjects $T_x E^r$ onto $T_y (\bC^{(4)} \times (S^1 \setminus \{0,\pi\}))$.
By construction, $T_x \widetilde{R}$ is the preimage of $T_y R$ under $d (\ev)_x$.
Also by construction, $\ev$ is a $C^r$ diffeomorphism from $p^{-1}(f)$ to $\gamma^{(4)} \times (S^1 \setminus \{0,\pi\})$.
Hence $T_x p^{-1}(f)$ is the preimage of $T_y (\gamma^{(4)} \times (S^1 \setminus \{0,\pi\}))$ under $d (\ev)_x$.
On the other hand, if $\phi : V \to W$ is a surjective linear map, and if $A$ and $B$ are subspaces of $W$, then $A+B = W$ if and only if $\phi^{-1}(A) + \phi^{-1}(B) = V$.
It follows that $x$ is a transverse point of intersection between $p^{-1}(f)$ and $\widetilde{R}$ if and only if $y$ is a transverse point of intersection between $\gamma^{(4)} \times (S^1 \setminus \{0,\pi\})$ and $R^*$.
The first condition holding for all $x \in p^{-1}(f) \cap \widetilde{R}$ is equivalent to (b), while the second condition holding for all $y \in ( \gamma^{(4)} \times (S^1 \setminus \{0,\pi\})) \cap R^*$ is equivalent to (c).
\end{proof}

\begin{prop}
\label{prop:E1}
The space $\Emb^\infty_1$ is open and dense in $\Emb^\infty$ in the $C^\infty$ topology.
\end{prop}

\begin{proof}[Proof of Proposition \ref{prop:E1}]
Fix $r \in \bZ^+$ and apply Lemma \ref{lem:transverse1}.
Because $\Emb^r_1 \subset \Emb^r$ is open in the $C^r$ topology, and because $\Emb^\infty$ is $C^r$-dense in $\Emb^r$, it follows that $\Emb^\infty_1 = \Emb^\infty \cap \Emb^r_1$ is  $C^r$-dense in $\Emb^r_1$ and that $\Emb^\infty_1$ is open in the $C^\infty$ topology.
Because $\Emb^r_1$ is  $C^r$-dense in $\Emb^r$, it follows that $ \Emb^\infty_1$ is  $C^r$-dense in $\Emb^r$.
Therefore, $\Emb^\infty_1 \subset \Emb^\infty$ is  $C^r$-dense in $\Emb^\infty$.
Since $r$ was arbitrary, it follows that $\Emb^\infty_1$ is $C^\infty$-dense in $\Emb^\infty$.
\end{proof}

\noindent
{\em Remark.} Proposition \ref{prop:E1} suffices in place of the Genericity Theorem in the proof of the Derivative Property for the spectral invariants.

\subsection{Binormals.}
\label{sec:binormals}
We begin by identifying binormals of $\gamma$ with critical points of $H | L_0$ away from $\Delta(\gamma)$, where $L_0 = \gamma \times \gamma$.
In parallel to the previous section, we show that for any smooth curve $\gamma$, the locus $H^{-1}(0) = \Delta(\gamma) \subset L_0$ is a submanifold of the zero set of $d H$.
We then show that for an open dense set of parametrized smooth Jordan curves, $H | L_0$ is a Morse-Bott function, and Morse away from $\Delta(\gamma)$.
We must then argue that for such a smooth Jordan curve, the set $\Crit(H) \times \{0,\pi\}$ compactifies $\cL_0 \cap \cL_1$ to a smooth manifold in a neighborhood of $L_0 \times \{0 , \pi\}$.
In order to do so, we work with a Weinstein neighborhood $N(L_0)$ of $L_0$ in $\bC^2$.
(The use of symplectic geometry may seem a little surprising here.)
By a rescaling method, we replace the intersection of $\cL_1$ with $N(L_0) \times (-\epsilon,\epsilon)$ by a smooth submanifold $X \subset N(L_0) \times (-\epsilon,\epsilon)$.
The submanifold $X$ intersects $\cL_0$ cleanly in $\Delta(\gamma) \times (-\epsilon,\epsilon)$, transversely away from $\Delta(\gamma) \times (-\epsilon,\epsilon)$, and satisfies $\cL_0 \cap X = \Crit(H) \times \{ 0 \} \sqcup (\cL_0 \cap \cL_1 \cap (N(L_0) \times (-\epsilon,\epsilon)))$.
The same applies in $N(L_0) \times (\pi-\epsilon,\pi+\epsilon)$.
In this way, we obtain the desired smooth manifold structure of $\cR(\gamma)$ in a neighborhood of $\gamma \times \gamma \times \{0,\pi\}$, for an open dense subset of parametrized Jordan curves.

\begin{lem}
\label{lem:binormal}
A point $(z_1,z_2) \in S^1 \times S^1 \setminus \Delta(S^1)$ is a critical point of $H \circ (f \times f)$ if and only if $(f(z_1),f(z_2)) \in \gamma \times \gamma \setminus \Delta(\gamma)$ is a binormal of $\gamma$.
\end{lem}

\begin{proof}
Suppose that $(z_1,z_2) \in S^1 \times S^1 \setminus \Delta(S^1)$.
The differential of $H \circ (f \times f)(z_1,z_2) = \frac14 |f(z_1)^2 - f(z_2)|^2$ is equal to 
\[
\left( \frac12 \langle f'(z_1), f(z_1)-f(z_2) \rangle, -\frac12 \langle f'(z_2), f(z_1)-f(z_2) \rangle \right)
\]
in coordinates\footnote{Recall that a \emph{coordinate} on $X$ is, among other things, a function $X \longrightarrow \bR$.} $(\partial/\partial z_1,\partial/\partial z_2)$ on $T^*_{(z_1,z_2)} (S^1 \times S^1)$.
Because $f$ is an injection and $z_1 \ne z_2$, the differential vanishes iff $f'(z_1)$ and $f'(z_2)$ are both orthogonal to $z_1-z_2$.
Because $f'$ is nonvanishing, this is equivalent to $T_{f(z_1)} \gamma$ and $T_{f(z_2)} \gamma$ both being perpendicular to $\overline{f(z_1) f(z_2)}$.
\end{proof}

If $M$ is a manifold and $f: M \to \bR$ is a $C^2$ function, then $f$ is {\em Morse} if the graph of $df$ intersects the 0-section transversely in $T^* M$, and $f$ is {\em Morse-Bott} if the graph of $df$ intersects the 0-section cleanly.
These notions are usually defined with reference to the nondegeneracy of the Hessian, but they are most useful to us in this form; cf. \cite[II.6.1]{gg1973}.

\begin{lem}
\label{lem:clean2}
For all $f \in \Emb^2(S^1,\bC)$, the copy of $\Delta(S^1)$ in the $0$-section of $T^* (S^1 \times S^1)$ is a clean component of intersection between the graph of $d (H \circ (f \times f) )$ and the $0$-section.
\end{lem}
\noindent
Equivalently, the restriction of $H$ to $\gamma \times \gamma \subset \bC^2$, $\gamma = f(S^1)$, is Morse-Bott in a neighborhood of $\Delta(\gamma)$.

\begin{proof}
The proof is a computation.
We proceed as in the proof of Lemma \ref{lem:binormal}.
In coordinates  $(z_1,z_2,\xi_1=\partial/\partial z_1,\xi_2=\partial/\partial z_2)$ on $T^* (S^1 \times S^1)$, the graph of $d H \circ (f \times f)$ is equal to the set of points whose $\xi_i$-coordinate is given by the $z_i$-partial derivative of $H \circ (f \times f)(z_1,z_2) = \frac14 |f(z_1)-f(z_2)|^2$.
That is, it equals the set of points
\[
\left( z_1,z_2, \frac12 \langle f'(z_1), f(z_1)-f(z_2) \rangle, -\frac12 \langle f'(z_2), f(z_1)-f(z_2) \rangle \right),
\]
where $z_1, z_2 \in S^1$.
In coordinates $(d z_1, d z_2, d \xi_1, d \xi_2)$ on the tangent space $T \, T^*(S^1\times S^1)$, the tangent space to the graph at this point is spanned by the tangent vectors obtained by differentiating the coordinates of this point with respect to $z_1$ and $z_2$: they are the rows of the matrix
\begin{eqnarray*}
\left(
\begin{matrix}
1 & 0 & \frac12 \langle f''(z_1), f(z_1)-f(z_2) \rangle + \frac12| f'(z_1)|^2 & -\frac12 \langle f'(z_1), f'(z_2) \rangle \\
0 & 1 & -\frac12 \langle f'(z_1), f'(z_2) \rangle & -\frac12 \langle f''(z_2), f(z_1)-f(z_2) \rangle +\frac12 | f'(z_2) |^2
\end{matrix}
\right).
\end{eqnarray*}
At a point $(z,z,0,0)$ on the graph, this simplifies to
\begin{eqnarray*}
\left(
\begin{matrix}
1 & 0 & \frac12 |f'(z)|^2 & -\frac12|f'(z)|^2 \\
0 & 1 & -\frac12|f'(z)|^2  & \frac12|f'(z)|^2
\end{matrix}
\right).
\end{eqnarray*}
Meanwhile, the tangent space to the 0-section at $(z,z,0,0)$ is spanned by the rows of the matrix
\[
\left(
\begin{matrix}
1 & 0 & 0 & 0 \\
0 & 1 & 0 & 0
\end{matrix}
\right),
\]
while the tangent space to the copy of $\Delta(S^1)$ in the 0-section at $(z,z,0,0)$ is spanned by the single row vector
\[
\left(
\begin{matrix}
1 & 1 & 0 & 0
\end{matrix}
\right).
\]
The value $f'(z)$ is nonzero, because $f$ is an embedding.
It follows that the intersection of the first two tangent spaces coincides with the third at $(z,z,0,0)$, for all $z \in S^1$.
This gives the result.
\end{proof}

For all $r \in \bZ^+$, $r \ge 2$, define 
\[
\Emb^r_2 = \{ f \in \Emb^r : H \circ (f \times f): S^1 \times S^1 \to \bR \textup{ is Morse-Bott and is Morse away from } \Delta(S^1) \}.
\]
Equivalently, the restriction of $H$ to $\gamma \times \gamma$, $\gamma = f(S^1)$, is Morse-Bott, and it is Morse away from $\Delta(\gamma)$.
By Lemma \ref{lem:clean2}, this is the same as
\[
\Emb^r_2 = \{ f \in \Emb^r : H \circ (f \times f): S^1 \times S^1 \to \bR \textup{ is Morse away from } \Delta(S^1) \}.
\]
Define
\[
\Emb^\infty_2 = \Emb^\infty \cap \Emb^r_2,
\]
and notice that it is independent of the choice of $r$.

\begin{prop}
\label{prop:E2}
$\Emb^\infty_2$ is open and dense in $\Emb^\infty$ in the $C^\infty$ topology.
\end{prop}

\begin{proof}
We shall prove that $\Emb^r_2$ is open and dense in $\Emb^r(S^1,\bC)$ in the $C^r$ topology, for all $r \in \bZ^+$, $r \ge 2$.
Since $\Emb^\infty_2 = \Emb^\infty \cap \Emb^r_2$, the proposition follows by applying the proof of Proposition \ref{prop:E1}, changing each ``1" to a ``2".

Consider the map
\begin{eqnarray*}
e: \Emb^r(S^1,\bC) \times (S^1)^{(2)} &\longrightarrow& T^*(S^1 \times S^1) \\
(f,z_1,z_2) &\longmapsto& d(H \circ (f \times f))(z_1,z_2).
\end{eqnarray*}

We claim that $e$ is a submersion.
Because $e$ is a bundle map over $(S^1)^{(2)}$, it suffices to check that the differential of $e$ at each point $(f,z_1,z_2)$ in the domain surjects onto the tangent space to the cotangent fiber over its image.
In coordinates $(\xi_1,\xi_2) =(\partial/\partial z_1, \partial/\partial z_2)$ on $T^*_{(z_1,z_2)}(S^1 \times S^1)$, we have
\begin{equation}
\label{e:e}
e(f,z_1,z_2) = \left(\frac{\del}{\del z_1} \frac{|f(z_1)-f(z_2)|^2}4, \frac{\del}{\del z_2} \frac{|f(z_1)-f(z_2)|^2}4 \right).
\end{equation}
Use coordinates $(d \xi_1, d \xi_2)$ on the tangent space of the fiber, and select an arbitrary tangent vector $(a,b)$ in these coordinates.
We argue that there exists a tangent vector $X \in T_f \Emb^r(S^1,\bC)$ such that the projection of $de(X)$ onto the tangent space to the fiber is $(a,b)$.
Recall that $X$ is a $C^r$ vector field along $f$.
The value of $de(X)$ in the component $d \xi_1$ is computed as 
\[
\lim_{t \to 0} \frac1t  \left( \frac{\del}{\del z_1} \frac{|(f+tX)(z_1) - (f+tX)(z_2)|^2}4 - \frac{\del}{\del z_1} \frac{|f(z_1) - f(z_2)|^2}4 \right), 
\]
which simplifies to
\[
\frac12 \langle f(z_1)-f(z_2), X'(z_1) \rangle+ \frac12 \langle X(z_1)-X(z_2), f'(z_1) \rangle.
\]
Similarly, the value in the component $d \xi_2$ equals
\[
-\frac12 \langle f(z_1)-f(z_2), X'(z_2) \rangle - \frac12 \langle X(z_1)-X(z_2), f'(z_2) \rangle.
\]
Setting $X(z_1) = X(z_2) = 0$, these become
\[
\frac12 \langle f(z_1)-f(z_2), X'(z_1) \rangle \quad \textup{and} \quad -\frac12 \langle f(z_1)-f(z_2), X'(z_2) \rangle.
\]
Because $f(z_1) \ne f(z_2)$, we can pick values and $X'(z_1)$ and $X'(z_2)$ to make these values equal to $a$ and $b$, respectively.
Finally, we may extend the given values (using bump functions) to produce the required $X \in T_f \Emb^r(S^1,\bC)$.
This completes the verification that $e$ is a submersion.

As in the proof of Lemma \ref{lem:transverse1}, the set of $f \in \Emb^r(S^1,\bC)$ such that the graph of $d(H \circ (f \times f))$ is transverse to the 0-section $(S^1)^{(2)} \subset T^* ((S^1)^{(2)})$ can be identified with the set of regular values of the map obtained by composing the inclusion and projection maps $e^{-1}((S^1)^{(2)}) \into \Emb^r(S^1,\bC) \times (S^1)^{(2)} \longrightarrow \Emb^r(S^1,\bC)$.
As before, this map is a Fredholm map between Banach manifolds, and the Sard-Smale theorem shows that this set is dense in $\Emb^r(S^1,\bC)$ in the $C^r$ topology.
The transversality condition is an open condition in $C^r(S^1,\bC)$ for all $r \ge 2$.
Therefore, the set of $f \in \Emb^r(S^1,\bC)$ such that $H \circ (f \times f)$ is Morse away from $\Delta(S^1)$ is an open, dense set.
This establishes the claim about $\Emb^r_2(S^1,\bC)$.
\end{proof}

\begin{lem}
\label{lem:smooth2}
Suppose that $f \in \Emb^\infty_2$ and $\gamma = f(S^1) \subset \bC$.
Then there exists $\epsilon > 0$ such that 
\[
\cR_\epsilon(\gamma) := \{ (z,w,t) \in \gamma \times \gamma \times (-\epsilon,\epsilon) : t=0 \textup{ and } dH(z,w) = 0, \textup{ or else } t \ne 0 \text{ and }\phi_t(z,w) \in \gamma \times \gamma \} 
\]
is a smooth submanifold of $\gamma \times \gamma \times (-\epsilon,\epsilon)$.
One component of $\cR_\epsilon(\gamma)$ is the annulus $\Delta(\gamma) \times (-\epsilon,\epsilon)$, and every other component is 1-dimensional and contains a unique point of the form $(z,w,0)$, where $z \ne w$ and $d H(z,w) = 0$.
\end{lem}

\noindent
The same result and proof holds with $(\pi-\epsilon,\pi+\epsilon)$ in place of $(-\epsilon,\epsilon)$.
Lemma \ref{lem:smooth2} thereby asserts that $\Crit(H) \times \{0,\pi\} \subset \gamma \times \gamma \times S^1$ compactifies $\cL_0 \cap \cL_1$ under the assumption that $f \in \Emb^\infty_2$.

\begin{proof}
Write $L = L_0 =  \gamma \times \gamma$.
Let $N(L) \subset \bC^2$ denote a Weinstein neighborhood of $L$, and write
\[
\Phi : (N(L),\omega | N(L)) \to (T^* L, d \lambda_\mathrm{can})
\]
for a symplectomorphism from $N(L)$ onto its image with $\Phi(q) = q$ for all $q \in L$, identifying $L$ with the 0-section of $T^* L$.
There exists a value $\epsilon > 0$ such that for all $t \in (-\epsilon,\epsilon)$, the displaced Lagrangian $\phi^H_t(L)$ is contained in $N(L)$ and maps to the graph of an exact 1-form in $T^* L$ under $\Phi$.
A primitive of this 1-form, also known as a generating function, is given as follows.
Define
\begin{eqnarray*}
S_t(q) : L \times (-\epsilon,\epsilon) &\longrightarrow& \bR, \\
(q,t) &\longmapsto& \int_0^t (H(\gamma_q(s)) - \lambda_\mathrm{can} (\dot{\gamma_q}(s))) ds.
\end{eqnarray*}
Here $\gamma_q(s)$ denotes the trajectory of the Hamiltonian vector field $X_H$ with the initial condition $\gamma_q(0) = q$.
Note that the function $S_t(q)$ is a version of symplectic action, and it coincides with the action when $q \in L \cap \phi^H_t(L)$.
It is smooth in $(t,q)$.
By \cite[Proposition 9.4.2]{mcduffsalamon}, we have $\Phi(\phi_t(L)) = \mathrm{graph}(d S_t)$.

Consider the smooth function
\begin{eqnarray*}
F_t(q) : L \times (-\epsilon,\epsilon) &\longrightarrow& \bR, \\
(q,t) &\longmapsto& \int_0^1 S'_{st}(q) ds.
\end{eqnarray*}
Its value at $t \ne 0$ is
\[
F_t(q) =  \int_0^1 S'_{st}(q) ds = \int_0^1  \frac1t  \frac{d}{ds} S_{st}(q) ds =  \frac1t \frac d{ds} \int_0^1 S_{st}(q) ds = \frac1t S_t(q).
\]
Its value at $t=0$ is
\[
F_0(q) = S'_0(q) = \left. \frac{d S_t(q)}{dt} \right|_{t=0}
= H(\gamma_q(0)) - \lambda_\mathrm{can}(\dot{\gamma_q}(0))
= H(q):
\]
the first two equations are definitions, the third applies the Fundamental Theorem of Calculus, and the fourth substitutes $\gamma_q(0) = q$ and $\lambda_\mathrm{can} | L \equiv 0$.
It follows that the critical points of $F_t : L \to \bR$ coincide with the critical points of $S_t$ when $t\ne0$ and of $H$ when $t=0$.

Form the locus
\[
X = \bigcup_{-\epsilon < t < \epsilon} \graph(d F_t) \times \{t\} \subset T^* L \times (-\epsilon,\epsilon).
\]
By what we argued about the critical points of $F_t(q)$, the intersection of $X$ with $L \times (-\epsilon,\epsilon) \subset T^*L \times (-\epsilon,\epsilon)$ is equal to
\[
\{ (q,t) : t = 0 \textup{ and } dH(q) = 0, \textup{ or } t \ne 0 \textup{ and } dS_t(q) = 0 \}.
\]
The mapping $\Phi \times \mathrm{Id} : N(L) \times (-\epsilon,\epsilon) \to T^* L \times (-\epsilon,\epsilon)$ carries $L \times (-\epsilon,\epsilon)$ to itself and $\cR_\epsilon(\gamma)$ into this intersection.
To prove the desired statement about $\cR_\epsilon(\gamma)$, it suffices to prove the analogous statement about this intersection.

The locus $X$ is a smooth submanifold, because $F_t(q)$ is smooth.
Its intersection with $L \times (-\epsilon,\epsilon)$ contains $\Delta(\gamma) \times (-\epsilon,\epsilon)$.
We claim that it is a clean component of intersection.
The tangent spaces to $X$, $L \times (-\epsilon, \epsilon)$, and $\Delta(\gamma) \times (-\epsilon,\epsilon)$ at $(z,z,0)$ are just the tangent spaces to $\graph(dH)$, the 0-section $L$, and $\Delta(\gamma)$, each multiplied against $T_0 (-\epsilon,\epsilon)$.
By Lemma \ref{lem:clean2}, this checks the cleanliness of the intersection of $X$ and $L \times (-\epsilon,\epsilon)$ along $\Delta(\gamma) \times (-\epsilon,\epsilon)$ at each point $(z,z,0)$.
Lemma \ref{lem:cleanlocus} checks the cleanliness of the intersection at each point $(z,z,t)$ with $0 < |t| < \epsilon$.

The assumption that $f \in \Emb^\infty_2$ means that $F_0 = H | L$ is Morse-Bott and is Morse away from $\Delta(\gamma)$.
Hence the intersection of $\graph(d F_0)$ with $L \subset T^*L$ is transverse away from $\Delta(\gamma)$.
It follows that the intersection of $\graph(d F_0)$ with $L \times (-\epsilon,\epsilon) \subset T^* L \times (-\epsilon,\epsilon)$ is transverse away from $\Delta(\gamma) \times \{0\}$.
It follows that by lowering the value $\epsilon$, the intersection of $X$ with with $L \times (-\epsilon,\epsilon) \subset T^* L \times (-\epsilon,\epsilon)$ is transverse away from $\Delta(\gamma) \times (-\epsilon,\epsilon)$, hence a smooth 1-manifold.
Moreover, each component contains a unique point $(q,0)$ with $q \notin \Delta(\gamma)$ and $dH(q)=0$.
\end{proof}

\begin{proof}[Proof of \Cref{thm:genericity}]
Define $\Emb^\infty_\mathrm{reg} = \Emb^\infty_1 \cap \Emb^\infty_2$.
By Propositions \ref{prop:E1} and \ref{prop:E2}, it follows that $\Emb^\infty_\mathrm{reg}$ is open and dense in $\Emb^\infty$ in the $C^\infty$ topology.
Let $f \in \Emb^\infty_\mathrm{reg}$ and write $\gamma = f(S^1)$.
Because $f \in \Emb^\infty_1$, Corollary \ref{cor:niceintersection} ensures that $\cR(\gamma)$ is a smooth 1-manifold away from $\gamma \times \gamma \times \{0,\pi\}$.
Because $f \in \Emb^\infty_2$, Lemma \ref{lem:smooth2} ensures that $\cR(\gamma)$ is a smooth 1-manifold nearby $\gamma \times \gamma \times \{0\}$.
The same result, \emph{mutatis mutandis}, shows that it is a smooth 1-manifold nearby $\gamma \times \gamma \times \{ \pi \}$, as well.
Thus, $\cR(\gamma)$ is a smooth, closed 1-manifold.
\end{proof}

\bibliographystyle{amsplain}
\bibliography{works-cited.bib}

\providecommand{\bysame}{\leavevmode\hbox to3em{\hrulefill}\thinspace}
\providecommand{\MR}{\relax\ifhmode\unskip\space\fi MR }
\providecommand{\MRhref}[2]{%
  \href{http://www.ams.org/mathscinet-getitem?mr=#1}{#2}
}
\providecommand{\href}[2]{#2}
\begin{thebibliography}{10}

\bibitem{asanoike}
Tomohiro Asano and Yuichi Ike, \emph{The rectifiable rectangular peg problem},
  {\tt arxiv:2412.21057} (2024).

\bibitem{Bieberbach1915}
Ludwig Bieberbach, \emph{{\"U}ber eine extremaleigenschaft des kreises},
  Jahresber. Dtsch. Math.-Ver. \textbf{24} (1915), 247--250.

\bibitem{gg1973}
Martin Golubitsky and Victor Guillemin, \emph{Stable mappings and their
  singularities}, Graduate Texts in Mathematics, vol.~14, Springer-Verlag, New
  York, 1973.

\bibitem{greenelobb1}
Joshua~Evan Greene and Andrew Lobb, \emph{The rectangular peg problem}, Ann. of
  Math. (2) \textbf{194} (2021), no.~2, 509--517.

\bibitem{greenelobb2}
\bysame, \emph{Cyclic quadrilaterals and smooth {J}ordan curves}, Invent. Math.
  \textbf{234} (2023), no.~3, 931--935.

\bibitem{greenelobb3}
\bysame, \emph{Floer homology and square pegs}, {\tt arXiv:2404.05179} (2024).

\bibitem{greenelobb4}
\bysame, \emph{Square pegs betweeen two graphs}, to appear in Comment. Math.
  Helv. (2024).

\bibitem{leczap}
R\'{e}mi Leclercq and Frol Zapolsky, \emph{Spectral invariants for monotone
  {L}agrangians}, J. Topol. Anal. \textbf{10} (2018), no.~3, 627--700.

\bibitem{mcduffsalamon}
Dusa McDuff and Dietmar Salamon, \emph{Introduction to symplectic topology},
  third ed., Oxford Graduate Texts in Mathematics, Oxford University Press,
  Oxford, 2017.

\bibitem{meyerson1981}
Mark~D. Meyerson, \emph{Balancing acts}, Topology Proc. \textbf{6} (1981),
  no.~1, 59--75 (1982).

\bibitem{pak}
Igor Pak, \emph{The discrete square peg problem}, {\tt arxiv.org/0804.0657}
  (2008).

\bibitem{schnirelman1929}
Lev Schnirleman, \emph{On some geometric properties of closed curves (in
  {R}ussian)}, Usp. Mat. Nauk \textbf{10} (1929), 34--44.

\bibitem{schwartz2020}
Richard~Evan Schwartz, \emph{A trichotomy for rectangles inscribed in {Jordan}
  loops}, Geometriae Dedicata \textbf{208} (2020), 177--196.

\bibitem{toeplitz1911}
Otto Toeplitz, \emph{Ueber einige {A}ufgaben der {A}nalysis situs},
  Verhandlungen der {S}chweizerischen {N}aturforschenden {G}esellschaft (1911),
  no.~4, 197.

\end{thebibliography}
\end{document}